\documentclass[a4paper,10pt]{article}
\usepackage{epsfig}
\usepackage{amsmath,amsthm, amssymb, amsfonts}
\usepackage{graphicx}
\usepackage{color}
\renewcommand{\baselinestretch}{1}

\newcommand{\singlespacing}{\let\CS=\@currsize\renewcommand{\baselinestretch}{1}\tiny\CS}
\newtheorem{thm}{Theorem}

\newtheorem*{prof*}{Proof}

\newcommand{\btb}{\begin{table}}
\newcommand{\etb}{\end{table}}
\newcommand{\bt}{\begin{Theorem}}
\newcommand{\et}{\end{Theorem}}
\newcommand{\bp}{\begin{Proof}}
\newcommand{\ep}{\end{Proof}}
\newcommand{\bl}{\begin{Lemma}}
\newcommand{\el}{\end{Lemma}}
\newcommand{\br}{\begin{Remark}}
\newcommand{\er}{\end{Remark}}
\newcommand{\bi}{\begin{itemize}}
\newcommand{\ei}{\end{itemize}}
\newcommand{\bean}{\begin{eqnarray*}}
\newcommand{\eean}{\end{eqnarray*}}
\newcommand{\be}{\begin{equation}}
\newcommand{\ee}{\end{equation}}
\newcommand{\ben}{\begin{equation*}}
\newcommand{\een}{\end{equation*}}
\newcommand{\disp}{\displaystyle}
\newcommand{\bc}{\begin{center}}
\newcommand{\ec}{\end{center}}

\newcommand\correspondingauthor{\thanks{Corresponding author} }

\begin{document}
\baselineskip=24pt
\parskip = 10pt

\def \qed {\hfill \vrule height5pt width 5pt depth 5pt}
\newcommand{\ve}[1]{\mbox{\boldmath$#1$}}
\newcommand{\IR}{\mbox{$I\!\!R$}}

\title{ Asymptotic properties of the volatility estimator from high frequency data modeled by mixed fractional Brownian motion }

\author{Ananya Lahiri \correspondingauthor \\ Chennai Mathematical Institute\\H1, SIPCOT IT Park, Siruseri\\ Kelambakkam, Tamil Nadu\\ India 603103 \\
email: ananya.isi@gmail.com }
\providecommand{\keywords}[1]{\textbf{\textit{keywords}} #1}

\date{}
\maketitle
\begin{abstract}
Properties of mixed fractional Brownian motion has been discussed by Cheridito (2001) and Zili (2006). We have proposed an estimator of volatility parameter for a model driven by MFBM. In our article we have shown that the estimator has some desirable asymptotic properties.
\end{abstract}

\keywords:{ mixed fractional Brownian motion; volatility estimator; high frequency data; quadratic variation; strong consistency; asymptotic normality; Berry Esseen bound; Malliavin calculus}

\section{Introduction}
In recent literature mixed fractional Brownian motion (MFBM), as a substitute for Brownian motion, has been used to construct finance models. Loosely speaking MFBM is linear combination of Brownian motion and fractional Brownian motion. Classical Black Scholes stock price model uses Brownian motion as ingredient stochastic process. Brownian motion has a crucial property that it's increment process has independence structure.  
But in reality it is not very appropriate assumption having independent stock price increment. So several attempts have been made to address this issue. Among many other attempts one of the
method is to replace Brownian motion in Black Scholes model by other stochastic processes like fractional Brownian motion or MFBM and then find the desirable properties suitable for financial applications. In this context Sun (2013) has discussed about the option pricing for Black Scholes market with MFBM. In that paper an estimator for volatility has also been proposed.

Regarding MFBM there is an article by Cheridito (2001) where properties of MFBM has been discussed. Zili (2006) too has discussed some properties of MFBM. There is an interesting result in Cheridito's article. We know that Brownian motion is semi martingale and fractional Brownian motion is not even a weak semi martingale. Cheridito (2001) has shown that for certain values of Hurst parameter MFBM is not weak semi martingale but for certain other values of Hurst parameter it is ``equivalent'' to Brownian motion. Two measures are equivalent means they are absolutely continuous with respect to each other. 

In this article we try to propose an estimator of volatility for a model from MFBM with $H\in(0,1)$ for high frequency data, in later section we will see what we mean by high frequency data. Though Sun (2013) has already proposed an estimator for the volatility for $H\in(\frac{3}{4},1)$, no properties of that estimator has been discussed in that article. 
Here we will provide the properties of the estimator.

Volatility estimator is essentially quadratic variation of the corresponding process for high frequency data. We know if ingredient process is Brownian motion then one can get central limit theorem for it's quadratic variation, which is an well established result. Central limit theorem for quadratic variation of fractional Brownian motion is available when Hurst parameter of the  fractional Brownian motion $H\in(0,\frac{3}{4})$. We also know from literature that there can not be any central limit theorem for quadratic variation of fractional Brownian motion with $H\in(\frac{3}{4},1)$. But in case of quadratic variation of MFBM we have observed that central limit theorem is possible for all values of Hurst parameter $H$, precisely $H\in(0,1)$, from contributory fractional Brownian motion. And the result is expected as we have seen in Cheridito's paper that for $H\in(\frac{3}{4},1)$ MFBM is equivalent to Brownian motion.

Another point to note is that for quadratic variation of Brownian motion there is no need for a normalising constant. But in case of MFBM for $H\in(\frac{1}{2},1)$ the normalising constant not only depends on sample size, say $N$, but also on the Hurst parameter $H$ from contributory fractional Brownian motion. The same normalising constant will also work for MFBM with $H\in(0,\frac{1}{2})$ but one will see that the normalising constant tends to $1$ as $N\rightarrow\infty$. 

So our observations are following: If one wants to use MFBM with $H\in (\frac{3}{4},1)$ for modelling stock price then volatility estimator has nice asymptotic properties for $H\in(\frac{3}{4},1)$.  Though MFBM has dependent increment but for $H\in(\frac{3}{4},1)$ regime it is equivalent to a Brownian motion (from Cheridito's work) and in this article we also found that central limit theorem holds for MFBM with $H\in(\frac{3}{4},1)$. 

Rest of the paper is organised as follows. Section \ref{sec2} is recapitulation of MFBM. In section \ref{sec3} there is discussion about the volatility estimators. Section \ref{sec4} deals with properties of volatility estimator that we have considered and simulation studies. Section \ref{sec5} has concluding remarks. Background material for the proofs and the detail proof is given in the Appendix.

\section{Recapitulation of properties of mixed fractional Brownian motion (MFBM)\label{sec2}}
 
Mixed fractional Brownian motion (MFBM) is the linear combination of fractional Brownian motions (FBM) with different Hurst parameters. For our result let us consider a special case of MFBM as follows $M_t^H=\alpha B_t+\beta B_t^H$, where Brownian motion $B_t$ and fractional Brownian motion $B_t^H$ are defined on same complete probability space $(\Omega, \mathcal{F},\mathbb{P} )$ 
and are independent, $\alpha \in \mathbb{R}, \ \beta \in \mathbb{R}, (\alpha,\beta)\ne (0,0)$, $H\in (0,1)$ is Hurst parameter, $ t \in (0,\infty)$.

Take $\alpha=1$. Then Cheridito has shown the following: $M_t^H$ is not weak semimartingale if $H\in (0,\frac{1}{2})\cup (\frac{1}{2},\frac{3}{4}]$; it is equivalent to $\sqrt{1+\beta^2}$ times Brownian motion for $H=\frac{1}{2}$; and equivalent to Brownian motion if $H \in (\frac{3}{4},1)$, i.e. for $H \in (\frac{3}{4},1)$ MFBM is semimartingale. For detail discussion, see Cheridito (2001). From Cheridito's work we also know that for $H\in(0,\frac{1}{2}) $ probability limit of quadratic variation of MFBM upto fixed time is infinity where as for $H\in(\frac{1}{2},1)$ the same limit is finite.

The moments and other properties of $M_t^H$ can be found in Zili (2006) and Sun (2013). Here we are restating some of those properties of MFBM which will be needed for our purpose.
$E(M_t^H)=0$;
$cov(M_s^H,M_t^H)=\alpha^2(t\wedge s)+\disp \frac{\beta^2}{2}(t^{2H}+s^{2H}-|t-s|^{2H})$;  
$M_{ht}^H(\alpha, \beta)\stackrel{d}{=}M_t^H(\alpha h^{1/2},\beta h^H)$, $\stackrel{d}{=}$ means `same distribution'.

\section{A Model driven by MFBM, $H\in(0,1)$ and Volatility Estimator \label{sec3}}
 Let us consider the following model \begin{equation}
S_t =S_0 \exp(\mu t+\sigma(B_t+B_t^H)-\frac{1}{2}\sigma^2(t+t^{2H}))
\end{equation} for $H\in(0,1)$ where $S_0>0,\ \mu\in(-\infty,\infty)$ is the drift parameter and $ \sigma >0$ is called the volatility parameter. This type of model has been considered in literature by Xiao {\it et. al.} (2012), Sun, L (2013) for $H\in(\frac{3}{4},1)$ in some finance context. We are interested in an estimator of the volatility parameter of this process for $H\in(0,1)$. Xiao {\it et. al.} (2011) proposed some maximum likelihood estimator for MFBM parameter from different model.

First we will set $\mu=0$ in the model. For $0=t_0<t_1<\cdots<t_N=1$ the observed values of the process be $S_{t_j}, \ j=0,\cdots, N$ and $t_{j+1}-t_{j}=\disp\frac{1}{N}\ \forall \ j=0,\cdots, N-1$. We note that this is high frequency data as sample size increases the time difference between two consecutive data points decrease. Let us propose the estimator of volatility $\sigma ^2$ based on the observation set $\{S_{t_j}, \ j=0,\cdots, N\} $ as 
follows:
\begin{equation}
\hat{\sigma}^2=\frac{1}{N}(\frac{1}{N}+\frac{1}{N^{2H}})^{-1}\sum_{j=0}^{N-1}\Big(\log\frac{S_{\frac{j+1}{N}}}{S_{\frac{j}{N}}}\Big)^2
\end{equation} 

Note that the normalising factor $N(\disp\frac{1}{N}+\frac{1}{N^{2H}})$ is dependent of both $N$ and $H$ where $H\in(0,1)$. For $H>\frac{1}{2}$ this normalising factor $\rightarrow 1$ as $N\rightarrow\infty$ and for For $H<\frac{1}{2}$ this normalising factor $\rightarrow \infty$ as $N\rightarrow\infty$. So for $H\in(\frac{3}{4},1)$ normalising factor eventually be $1$ as $N\rightarrow \infty$.

We will see that this estimator converges almost surely and asymptotically normally as sample size $N$ increases. We also find the rate of convergence for the distribution.

\subsection{An estimator proposed by Sun (2013), $H>\frac{3}{4}$}
For $0\le T_1\le T_2$, if we have observed $S_t,\  T_1=t_0<t_1<\cdots<t_N=T_2$, then volatility estimator, say $\hat \sigma$, of the interval $[T_1,T_2]$ given by Sun, L (2013) for $H>\frac{3}{4}$ is \begin{equation}
\tilde\sigma^2=\frac{1}{T_2-T_1+T_2^H-T_1^H}\sum_{j=0}^{N-1}\Big(\log\frac{S_{t_{j+1}}}{S_{t_{j}}}\Big)^2,
\end{equation}
For this estimator no statistical properties has been discussed.

Let us look at estimator proposed by Sun (2013). Take special case as follows. Assume the time points $T_1=t_0<t_1<\cdots<t_N=T_2$ are equidistant, that is $t_{j+1}-t_{j}=\disp\frac{1}{N}\ \forall \ j=0,\cdots, N-1$. For simplicity take $T_1=0$ and $T_2=1$. We note that for this particular choice of $T_1, T_2$ the normalising factor $T_2-T_1+T_2^H-T_1^H$ proposed by Sun (2013) will be $2$ which is just a constant. For this special choice of time points, essentially for high frequency data case, \begin{equation}
\disp\frac{1}{N}(\frac{1}{N}+\frac{1}{N^{2H}})^{-1} \hat\sigma^2=2 \tilde\sigma^2.
\end{equation} So the estimator proposed by Sun (2013) can be obtained after a constant adjustment in our estimator. As we have derived the asymptotic properties for our estimator, using that one can get asymptotic properties of the estimator proposed by Sun (2013) too.

\section{Almost sure convergence, asymptotic normality, Berry Esseen bound for volatility estimator \label{sec4}}

\subsection{Main Results}
\begin{thm}
The estimator $\hat{\sigma}^2$ converges almost surely to $\sigma^2$.
\end{thm}
\begin{proof}
See Appendix: Almost sure convergence
\end{proof}

\begin{thm}
The estimator $\hat{\sigma}^2$ is asymptotically distributed as   
i.e. $ \disp\frac{\sqrt{N}(\hat\sigma^2-\sigma^2)}{\sqrt{c}}\mathop{\rightarrow}^{d} \mathcal{N}(0,1)$ as $N\rightarrow\infty$, where $c=2\sigma^4$. 
\end{thm}
\begin{proof}
See Appendix: Asymptotic normality
\end{proof}

\begin{thm}
Let $\disp F_N:=\sqrt{N}\frac{(\hat\sigma^2-\sigma^2)}{\sqrt{c}}$ and  $c=2\sigma^4$, 
then for all $x$,  
$\disp \sqrt{N}(P( F_N \leq x)-\Phi(x))\rightarrow -\frac{\Phi^{(3)}(x)}{3\sqrt{2}}$ as $N\rightarrow\infty$ where $\Phi(x)$ is cumulative normal distribution function and $\Phi^{(3)}$ is it's third derivative. 
\end{thm}
\begin{proof}
See Appendix: Berry Esseen Bound
\end{proof}
\subsection{Simulation Studies}
In this section We present the simulation result for MFBM driven model and it's estimator. We use somebm and dvfbm packages from R to simulate Brownian and fractional Brownian motion. We keep drift parameter $\mu=0$. We generate each sample paths with $N=1000$ points and replicate 200 times to find the mean and variance. We repeat the simulation for different values of $\sigma^2$ and for $H\in (0,1)$. Simulation shows that estimators are excellent for both $\disp H<\frac{1}{2}$ and $\disp H>\frac{1}{2}$ for high frequency data with 1000 values in the time scale 0 to 1.

As we know that for $\disp H>\frac{1}{2}$ the normalising factor $\disp\frac{1}{N}(\frac{1}{N}+\frac{1}{N^{2H}})^{-1}\rightarrow 1$ as $N\rightarrow\infty$ we have shown simulation for estimators without normalising factor in case of $\disp H>\frac{1}{2}$. We see there is bias for estimators but we add mean squared error along with variance to check how wide the bias is. Numerical study says they are close in terms of variance and mean squared error as $H$ is far from 0.5, near 0.5 the estimators are bad as expected.
   
\begin{table}
\caption{The MEAN, VAR of the estimators when $\sigma^2$= 0.4}
\begin{center}
\begin{tabular}{|c|c|c|c|c|c|} \hline
  \multicolumn{6}{|c|}{$\sigma^2$=0.4} \\  \hline
     H &   0.25 &   0.45 &   0.55 &   0.75  &      0.95 \\
   MEAN &   0.4016004   & 0.4002948     &   0.4024902   &   0.4020853    &       0.3991005    \\
   VAR & 	0.0003127	&	0.0003119	&	0.0003692	&	0.0003449	&	0.0003938	  \\ \hline
\end{tabular}
\end{center}

\caption{The MEAN, VAR of the estimators when $\sigma^2$= 1.6}
\begin{center}
\begin{tabular}{|c|c|c|c|c|c|} \hline
  \multicolumn{6}{|c|}{$\sigma^2$=1.6} \\  \hline
     H &   0.25 &   0.45 &   0.55 &   0.75  &     0.95 \\
   MEAN & 1.603423     &  1.598156    &   1.608824   &   1.609965    &   1.607485        \\
   VAR & 0.0062663		&	0.0052046	&	0.005706	&	0.0050978	&	0.0050397	  \\ \hline
\end{tabular}
\end{center}

\caption{The MEAN, VAR of the estimators when $\sigma^2$= 6.4}
\begin{center}
\begin{tabular}{|c|c|c|c|c|c|} \hline
  \multicolumn{6}{|c|}{$\sigma^2$=6.4} \\  \hline
     H &   0.25 &   0.45 &   0.55 &   0.75  &     0.95 \\
   MEAN &   6.428518   &   6.395258   &   6.432677   &   6.470972    &    6.429915       \\
   VAR & 	0.0972558	&	0.08041567	&	0.0756612	&	0.079025	&	0.07159175	  \\ \hline
\end{tabular}
\end{center}
\end{table}

\begin{table}
\caption{The MEAN, VAR, MSE of the estimators without normalising factor when $\sigma^2$=0.4 }
\begin{center}
\begin{tabular}{|c|c|c|c|} \hline
  \multicolumn{4}{|c|}{$\sigma^2$=0.4} \\  \hline
     H &      0.55 &   0.75  &     0.95 \\
   MEAN &  0.6016713   &   0.4135389   &   0.4013176    \\
    VAR & 0.0006716		&	0.00034487	&	0.0003427 \\ 
    MSE & 0.0413436		&	0.00052817	&	0.0003444  \\\hline
\end{tabular}
\end{center}

\caption{The MEAN, VAR, MSE of the estimators without normalising factor when $\sigma^2$= 1.6}
\begin{center}
\begin{tabular}{|c|c|c|c|} \hline
  \multicolumn{4}{|c|}{$\sigma^2$=1.6} \\  \hline
    H &      0.55 &   0.75  &     0.95 \\
   MEAN & 2.384892    &   1.657533   &    1.610953           \\
   VAR & 0.01041467		&	0.0051399	&	0.0049417  \\
   MSE & 0.6264697		&	0.00845	&	0.0050617  \\ \hline
\end{tabular}
\end{center}

\caption{The MEAN, VAR, MSE of the estimators without normalising factor when $\sigma^2$= 6.4}
\begin{center}
\begin{tabular}{|c|c|c|c|} \hline
  \multicolumn{4}{|c|}{$\sigma^2$=6.4} \\  \hline
     H &      0.55 &   0.75  &     0.95 \\
   MEAN &  9.657529   &  0.650319    &  6.471309        \\
   VAR & 	0.1756582	&	0.0895674	&	0.0786475  \\
   MSE & 	10.78715	&	0.1522243	&	0.0837325	\\ \hline
\end{tabular}
\end{center}
\end{table}
\section{Conclusion \label{sec5}}
We have shown that the estimator for volatility for high frequency data from model driven by MFBM converges almost surely i.e. strongly consistent and 
is asymptotically normally distributed. In the process of doing that we have shown the central limit theorem for quadratic variation of MFBM, $H\in(0,1)$ too. We have also found the rate of convergence for the distribution through Berry Esseen bound.
Note that unlike quadratic variation of pure fractional Brownian motion, where one get asymptotic normality for $H\in(0,\frac{3}{4})$, here for mixed fractional Brownian motion we get asymptotic normality for all values of $H$, precisely for $H\in(0,1)$. The result is expected as we have seen in Cheridito's paper that MFBM is equivalent to Brownian motion in regime $H\in(\frac{3}{4},1)$. 

\section{Appendix}
\subsection*{Notations and Background \label{note}}
\appendix\normalsize 
In this section we introduce only those notations and established results which will be needed for derivation of our results.

Our Brownian motion $B_t$ and fractional Brownian motion $B_t^H$ are both centered, continuous, mean zero real valued Gaussian processes with covariance functions $R_B(s,t)=cov(B_s,B_t)=t\wedge s$ and $R_B^H(s,t)=cov(B_s^H, B_t^H)=\frac{1}{2}[t^{2H}+s^{2H}-|t-s|^{2H}]$ respectively.  Let $\mathcal{E}$ be the set of real valued step functions. Let the isonormal Gaussian processes $\{B(\phi), \phi\in \mathcal{H}_1\big(L^2(\mathbb{R},dt)\big)\}$ and $\{B^H(\psi), \psi\in \mathcal{H}_2\big(\mbox{space of distributions on } \ \mathbb{R} \big)\}$ have been constructed from $B_t$ and $B^H_t$ where $\mathcal{H}_1$ and $\mathcal{H}_2$ are the associated Hilbert spaces for $B_t$ and $B_t^H$, which are the closure of $\mathcal{E}$ with inner product as corresponding covariances respectively. So, we have $E(B(I_{[0,s]}) B(I_{[0,t]}))=cov(B_s,B_t)=\langle I_{[0,s]},I_{[0,t]}\rangle_{\mathcal{H}_1}=t\wedge s$ and $E(B^H(I_{[0,s]}) B^H(I_{[0,t]}))=cov(B^H_s,B^H_t)=\langle I_{[0,s]},I_{[0,t]}\rangle_{\mathcal{H}_2}=\frac{1}{2}[t^{2H}+s^{2H}-|t-s|^{2H}]$.
Take $\alpha=1$ and $\beta=1$. That makes $M_t^H=B_t+ B_t^H$ and $cov(M_s^H,M_t^H)=(t\wedge s)+\disp \frac{1}{2}(t^{2H}+s^{2H}-|t-s|^{2H}) $.

For $\phi, \psi \in \mathcal{E}$ let us define inner product $\langle\phi,\psi\rangle_{\mathcal{E}}=\langle\phi,\psi\rangle_{\mathcal{H}_1}+\langle\phi,\psi\rangle_{\mathcal{H}_2}$. Define $\mathcal{H}$ be the Hilbert space given by closure of $\mathcal{E}$ with inner product $\langle\phi,\psi\rangle_{\mathcal{E}}$ for $\phi, \psi \in \mathcal{E}$. So we get 
$\langle\phi,\psi\rangle_{\mathcal{H}}=\langle\phi,\psi\rangle_{\mathcal{H}_1}+\langle\phi,\psi\rangle_{\mathcal{H}_2}$, for $\phi,\psi \in \mathcal{H}$. 
For $\phi=\sum_i a_i1_{[0,s_i]}$, set $ M(\phi)=\sum_i a_iM^H_{s_i}$. Similarly let $\psi=\sum_j b_j1_{[0,t_j]}$, set $ M(\psi)=\sum_j b_jM^H_{t_j}$. So, $E(M(\phi)M(\psi))=\langle\phi,\psi\rangle_{\mathcal{E}}$. Next for $\phi\in \mathcal{H}$, there are $\phi_n \in \mathcal{E}$ such that $\phi_n\rightarrow \phi$ in $\mathcal{H}$ then $M(\phi) $ is the $L^2$ limit of  $M(\phi_n) $, $M(\phi) $ is in $L^2( \Omega, \mathcal{F})$. 
From isometry we get $E(M(\phi)M(\psi))=\langle\phi,\psi\rangle_{\mathcal{H}}$ where $\phi, \psi \in \mathcal{H}$. 
$\{M(\phi), \phi \in \mathcal{H}\}$ is called isonormal Gaussian process for MFBM. Construction of isonormal Gaussian process from covariance  function, see Section 8.3, Peccati (2011).  

We also note that $\mathcal{H}$ can be written as $L^2(\mathbb{R},\mathcal{B},\nu)$ where $\nu$ is non atomic measure.

Let $H_n$ be $n$ th Hermite polynomial satisfying \begin{equation}
 \disp \frac{d}{dx}H_n(x)=H_{n-1}(x),\ n\geq 1.
\end{equation} Take $\phi \in \mathcal{H}$ such that $\|\phi\|_{\mathcal{H}}=1$. Consider random variables $H_n(M(\phi))$ and take the closure of the span of these random variables as a subspace of $L^2(\Omega, \mathcal{F})$. This subspace is the $n$ th order Wiener chaos $\mathcal{W}_n$. 

$I_n$, the multiple stochastic (Wiener Ito) integral with respect to isonormal Gaussian process $M$, is a map from $\mathcal{H}^{\odot n}$ to $\mathcal{W}_n$, $\mathcal{H}^{\odot n}$ being symmetric tensor product of $\mathcal{H}$.  $ \mathcal{H}^{\odot n}$ has norm $ \disp \frac{1}{\sqrt{n !}}\|.\|_{\mathcal{H}^{\otimes n}}$, $\mathcal{H}^{\otimes n}$ is tensor product of $\mathcal{H}$.

Then for $f\in \mathcal{H}^{\odot n}$,  we also have $I_n(f)=I_n(\tilde f)$,  $\tilde{f}$ is symmetrization of $f$.

For $\phi \in \mathcal{H},\ \disp I_n(\phi^{\otimes n})=n! H_n(I_1(\phi))= n! H_n(M(\phi))$ is linear isometry between $\mathcal{H}^{\odot n} $ and $\mathcal{W}_n$ by Proposition 8.1.2. Peccati (2011). 

Now for $f\in \mathcal{H}^{\odot n}$ and $g\in \mathcal{H}^{\odot m}$ we have followings:
\begin{eqnarray}
 E(I_n(f)I_m(g))&=&n!\langle\tilde{f},\tilde{g}\rangle _{\mathcal{H}^{\otimes n}}\ \mbox{ if } \ m=n\\
 E(I_n(f)I_m(g))&=&0\ \mbox{ if }\ m \neq n
\end{eqnarray}

Let $\{e_i, i\geq 1\}$ be an orthonormal basis of $\mathcal{H}$, $m,n \geq 1, \ r=0,\cdots, n\wedge m $. $f\otimes_r g \in \mathcal{H}^{\otimes (m+n-2r)}$ is contraction is defined as \begin{equation}
f\otimes_r g=\sum_{i1,\cdots,ir=1}^{\infty}\langle f, e_{i1}\otimes\cdots\otimes e_{ir}\rangle_{\mathcal{H}^{\otimes r}}\langle g, e_{i1}\otimes\cdots\otimes e_{ir}\rangle_{\mathcal{H}^{\otimes r}}. \label{cc} 
\end{equation} This definition does not depend on the choice of orthonormal basis and $\langle f, e_{i1}\otimes\cdots\otimes e_{ir}\rangle_{\mathcal{H}^{\otimes r}}\in \mathcal{H}^{\odot (n-r)}$, $\langle g, e_{i1}\otimes\cdots\otimes e_{ir}\rangle_{\mathcal{H}^{\otimes r}}\in \mathcal{H}^{\odot (m-r)}$. $f\otimes_r g $ is not necessarily symmetric. Let $f\tilde \otimes_r g$ is symmetrization of $f\otimes_r g$. Then by proposition (8.5.3), Peccati (2011) \begin{equation}\disp I_n(f)I_m(g)= \sum_{r=0}^{m \wedge n} r!  \Bigl(\begin{matrix} n\\ r \end{matrix} \Bigr) \Bigl(\begin{matrix} m\\ r \end{matrix} \Bigr) I_{n+m-2r}(f\tilde\otimes_r g ).\label{tp}
\end{equation}
Also for $n=m=r$ we have \begin{equation}
I_0(f\otimes_r g)=\langle f,g\rangle_{\mathcal{H}^{\otimes r}}.
\end{equation} 
The following is the hypercontractivity property for multiple Wiener Ito integrals, see equation 8.4.18, Peccati (2011) or lemma 2.1, Nourdin (2014). 
\begin{equation}
[E(I_n(f))^r]^{\frac{1}{r}}\leq (r-1)^{\frac{n}{2}}[E(I_n(f))^2]^{\frac{1}{2}}, \ r \geq 2 \label{hc}
\end{equation}
Let $F$ be a functional of the isonormal Gaussian process $M$ such that $E(F(M)^2)<\infty$ then there is unique sequence $f_n\in \mathcal{H}^{\odot n}$ and $F$ can be written as sum of multiple stochastic integrals as $ \disp F=\sum_{n\ge 0}I_n(f_n)$ with  and $I_0(f_0)=E(F)$ where the series converges in $L^2$, by Proposition 8.4.6, Peccati (2011).

For $\phi_1, \cdots, \phi_n \in \mathcal{H}$, let $F=g(M(\phi_1),\cdots,M(\phi_n))$ with $g$ smooth compactly supported and $\mathcal{S}$ is the collection of all smooth cylindrical random variables of the form $F$ defined above. Then Malliavin derivative $D$ is $\mathcal{H}$ valued random variable  
defined as follows: 
\begin{equation}
DF=\sum_{i=1}^n\frac{\partial g}{\partial x_i}(M(\phi_1),\cdots,M(\phi_n))\phi_i.
\end{equation} $DF$ is the element of $L^2(\Omega, \mathcal{H})$. By iteration the $m$ th derivative $D^mF$ can be defined and $D^mF$ is an element of $L^2(\Omega, \mathcal{H}^{\otimes m}) $. $\mathbb{D}^{m,2}$ denote the closure of $\mathcal{S}$ with respect to following norm $\disp \|F\|^2_{m,2}=E F^2+\sum_{i=1}^m E\|D^i F\|^2_{\mathcal{H}^{\otimes i}}$. 
As $\mathcal{H}$ is $ L^2(\mathbb{R},\mathcal{B},\nu)$ for non atomic measure $\nu$,  so $ DF$ can be identified with an element of $L^2(\Omega\times \mathbb{R}) $. 
Denote $DF=(D_tF)_{t\in \mathbb{R}}$ 
\begin{equation}
D_tF=\sum_{i=1}^n\frac{\partial g}{\partial x_i}(M(\phi_1),\cdots,M(\phi_n))\phi_i(t), \ t \in \mathbb{R}
\end{equation}
If $F=I_n(f),f \in \mathcal{H}^{\odot n}$, 
for every $t \in \mathbb{R}$, then \begin{equation}D_{t}F= D_{t}I_n(f)=n I_{n-1} f(.,t). \label{dd}
\end{equation} 
$I_{n-1}(f(.,t))$ means $n-1$ multiple stochastic integral is taken with respect to first $n-1$ variables $t_1,\cdots, t_{n-1}$ of $f(t_1,\cdots, t_{n-1},t), \ t$ is kept fixed.

If $\disp F=\sum_{n\ge 0}I_n(f_n)$, $f_n\in \mathcal{H}^{\odot n}$, then for every $t \in \mathbb{R}$, \begin{equation}D_{t}F= \sum_{n\ge 1}n I_{n-1} f_n(.,t). \label{dd1}
\end{equation} 

Let us introduce the divergence operator $\delta$ as adjoint of Malliavin derivative $D$. $u\in L^2(\Omega,\mathcal{H})$ is said to belongs to domain of $\delta$ $Dom(\delta)$ iff 
\begin{equation}
|E\langle DF,u\rangle_{\mathcal{H}}\leq c_u\|F\|_{L^2}
\end{equation}
with $F\in \mathbb{D}^{1,2}$, $c_u$ constant depends on $u$ only. For $u\in Dom(\delta)$, for every $F\in \mathbb{D}^{1,2}$, $\delta(u)$ satisfies duality relationship \begin{equation}
E(F\delta(u))=E\langle DF, u\rangle_{\mathcal{H}}.
\end{equation} 
Orstein Uhlenbeck operator $L$ is defined as $LF=-\delta DF$ and for $F=I_n(f)$ with $f$ as before, $LF=-nF$.

For the ease of readability we are reproducing the following theorems with references. 

To prove asymptotic normality we will use the following two theorems [5.1] and [5.2] taken from Tudor C.A. (2008).

\begin{thm}
Let $I_n(f) $ be a multiple integral of order $n\ge 1 $ with respect to an isonormal process $M$. Then 
$$ d(\mathcal{L}(I_n(f)),\mathcal{N}(0,1))\le c_n [E(|DI_n(f)|_{\mathcal{H}}^2-n)^2]^{\frac{1}{2}}$$ where $D$ is the Malliavin derivative with respect to $M$ and $\mathcal{H}$ is the canonical Hilbert space associated to $M$. Here $d$ can be any of the distances like Kolmogorov Smirnov distance, or total variation distance etc. and depending upon $d$ and the order $n$ one will end up a constant $c_n$. 
$\mathcal{L}(M)$ stands for law of $M$. \label{md}
\end{thm} 
\begin{thm}
Fix $n\ge 2$ and let $(F_k,k >1), \ F_k=I_n(f_k)$ (with $f_k \in \mathcal{H}^{\odot n}$, for every $k\ge 1$) be a sequence of square integrable random variables in the $n$th Wiener chaos of an isonormal process $M$ such that $E[F_k^2]^2\rightarrow 1$ as $k\rightarrow \infty$. Then the following are equivalent:
(i) The sequence $(F_k)_{k\ge 0}$ converges in distribution to the normal law $\mathcal{N}(0,1)$.
(ii) One has $E[F_k^4]\rightarrow 3$ as $k\rightarrow \infty$.
(iii) For all $1\le l \le n-1$ it holds that $\disp\lim_{k\rightarrow \infty}|f_k\otimes_l f_k|_{\mathcal{H}^{\otimes 2 (n-l)}}=0$.
(iv) $ |DF_k|^2_{\mathcal{H}}\rightarrow n \ \mbox{in}\ L^2 \ \mbox{as}\ k\rightarrow \infty$, where $D$ is the Malliavin derivative with respect to $M$.\label{fm}
\end{thm}
To prove Berry Esseen bound we will use the following two theorems from Nourdin and Peccati (2009).
\begin{thm}
Let $F \in D^{1,2}$ have zero mean and $N \sim \mathcal{N} (0, 1)$. Then,
$$d Kol(F,N) \leq \sqrt{ E[(1-\langle DF, -DL^{-1}F\rangle_{\mathcal{H}})^2]}.$$
If $F = I_n(f)$ for some $n \geq 2$ and $f \in  \mathcal{H}^{\otimes n}$, then $\disp\langle DF, -DL^{-1}F\rangle_{\mathcal{H}} = \frac{1}{n}\|DF\|^2_{\mathcal{H}}$ and therefore $$d Kol(F,N) \leq \sqrt{ E[(1-\langle DF, \frac{1}{n}\|DF\|^2_{\mathcal{H}})^2]}.$$\label{ou}
\end{thm}
\begin{thm}
Let $F_n$ , $n\geq 1$ be sequence of centered and squared integrable functional of Gaussian process $M=\{M(\phi);\phi\in\mathcal{H}\}$ such that $E(F_n^2)\rightarrow 1$ as $n\rightarrow\infty$. If the following three conditions held:

i) for every $n$, one has that $F_n\in \mathbb{D}^{1,2}$ and $F_n$ has an absolutely continuous law with respect to Lebesgue measure

ii) the quantity $\psi(n) =\sqrt{E[(1 - \langle D F_n,-D L^{−1}F_n\rangle_{\mathcal{H}})^2]}$ is such that: 

(a) $\psi(n)$ is finite for every $n$; 

(b) as $n\rightarrow\infty $, $\psi(n)$ converges to zero; 

and (c) there exists $m\geq 1$ such that $\psi(n) > 0$ for $n\geq m$;

iii) as $n\rightarrow\infty $ the two dimensional vector $\disp (F_n, \frac{ 1-\langle DF_n,-DL^{-1} F_n \rangle_{\mathcal{H}}}{\psi(n)})$ converges to distribution to a centered two-dimensional Gaussian vector $(N_1,N_2)$ such that $EN_1^2=EN_2^2=1$ and $EN_1N_2=\rho$. Then upper bound $d_{Kol}(F_n,N)\leq \psi(n)$ holds. Moreover, for every $z\in \mathbb{R}$ $$ \psi_n^{-1} [P (F_n \leq z )-\Phi(z) ]\rightarrow \frac{\rho}{3} \Phi(z)^{(3)}.$$\label{be}
\end{thm}

\subsection*{Proof for main results}
\appendix\normalsize 
This part is required for the proof of the theorems mentioned in sec \ref{sec4}.
Let us calculate the typical term present in the summand of the estimator as follows.
$$\disp\log\frac{S_{t_{j+1}}}{S_{t_{j}}}=\mu (t_{j+1}-t_j)-\frac{1}{2}\sigma^2 (t_{j+1}-t_j)- \frac{1}{2}\sigma^2 (t_{j+1}^{2H}-t_j^{2H})+\sigma (B_{t_{j+1}}-B_{t_{j}}) +\sigma (B^H_{t_{j+1}}-B^H_{t_{j}}) $$

\noindent For simplicity set $\mu=0,\  t_{j+1}=\disp\frac{j+1}{N}, \  t_j=\frac{j}{N}, \ T_1=0, \  T_2=T=1$. Also denote $ A_j=\disp t_{j+1}-t_j=\frac{1}{N},\  E_j=t_{j+1}^{2H}-t_j^{2H}=\frac{(j+1)^{2H}-j^{2H}}{N^{2H}}, \  C_j=B_{\frac{j+1}{N}}-B_{\frac{j}{N}}, \ D_j=B^H_{\frac{j+1}{N}}-B^H_{\frac{j}{N}} $.

\noindent Then \begin{eqnarray*}
&&\sum_{j=0}^{N-1}\Big(\log\frac{S_{\frac{j+1}{N}}}{S_{\frac{j}{N}}}\Big)^2\\
&=& \sum_{j=0}^{N-1}\frac{\sigma^4}{4}(A_j+E_j)^2  + \sum_{j=0}^{N-1} (-\sigma^3) (A_j+E_j)(C_j+D_j)  + \sum_{j=0}^{N-1} \sigma^2 (C_j+D_j)^2 \\
&=& U_1+U_2+U_3
\end{eqnarray*}
where $U_1, U_2, U_3$ are the terms respectively.
 
\noindent Now the first term is 
\begin{eqnarray*}
U_1&=&\sum_{j=0}^{N-1}\frac{\sigma^4}{4}(A_j+E_j)^2 \\
&=& \frac{\sigma^4}{4}\sum_{j=0}^{N-1} \Big(\frac{(j+1)^{2H}-j^{2H}}{N^{2H}}+\frac{1}{N}\Big)^2\\
&=& \frac{\sigma^4}{4}\sum_{j=0}^{N-1} \Big( \Big(\frac{(j+1)^{2H}-j^{2H}}{N^{2H}} \Big)^2 -\frac{2}{N} \Big(\frac{(j+1)^{2H}-j^{2H}}{N^{2H}}\Big) +\frac{1}{N^2} \Big)\\     
&\le & \frac{\sigma^4}{4}\sum_{j=0}^{N-1} \Big(\frac{ 4H^2 (j+1)^{4H-2}}{N^{4H}} +\frac{2. 2H(j+1)^{2H-1}}{N^{1+2H}}+\frac{1}{N^2}\Big)\\
&\le &  \frac{\sigma^4}{4} \Big(\frac{4H^2. N^{4H-1}}{N^{4H}}+\frac{4H. N^{2H}}{N^{1+2H}}+\frac{1}{N} \Big)\\
&=& \frac{\sigma^4}{4N}(2H+1)^2\rightarrow 0 \  \mbox{as}\  N \rightarrow \infty
\end{eqnarray*}
Note that even if we have kept $\mu$ as it is we would still get $U_1\rightarrow\ 0$ as $N \rightarrow \infty$. 
For our next step calculation we set $\alpha=1, \beta=1$ in $M_t^H$. For simplicity again we may use $M_t$ instead of $M_t^H$. 
The second term is 
\begin{eqnarray*}
U_2&=&\sum_{j=0}^{N-1} (-\sigma^3) (A_j+E_j)(C_j+D_j) 
=-\sigma^3 \sum_{j=0}^{N-1} \Big( \frac{(j+1)^{2H}-j^{2H}}{N^{2H}}+\frac{1}{N}\Big)(M_{\frac{j+1}{N}}-M_{\frac{j}{N}})
\end{eqnarray*}
We note that $EU_2=0$.
 
The variance of the increments of $M_t^H$ is as follows:
\begin{eqnarray*}
&& E(M_{\frac{j+1}{N}}-M_{\frac{j}{N}})^2
= E(B_{\frac{j+1}{N}}-B_{\frac{j}{N}})^2+E(B^H_{\frac{j+1}{N}}-B^H_{\frac{j}{N}})^2
=\frac{1}{N^{2H}}+\frac{1}{N}
\end{eqnarray*} 
For $k\ne j$ the covariance of the increments of $M_t^H$ is
\begin{eqnarray*}
&& E(M_{\frac{j+1}{N}}-M_{\frac{j}{N}})(M_{\frac{k+1}{N}}-M_{\frac{k}{N}})\\
&=& E(B^H_{\frac{j+1}{N}}-B^H_{\frac{j}{N}})(B^H_{\frac{k+1}{N}}-B^H_{\frac{k}{N}})+E(B_{\frac{j+1}{N}}-B_{\frac{j}{N}})(B^H_{\frac{k+1}{N}}-B^H_{\frac{k}{N}})+E(B_{\frac{k+1}{N}}-B_{\frac{k}{N}})(B^H_{\frac{j+1}{N}}-B^H_{\frac{j}{N}})\\
&=&\frac{1}{2}\Big[|\frac{k-j+1}{N}|^{2H}+|\frac{k-j-1}{N}|^{2H}-2|\frac{k-j}{N}|^{2H} \Big]+0+0
\end{eqnarray*}

The variance of the second term is $E(U_2^2)$ which is equal to
\begin{eqnarray*}
&&\sigma^6 \sum_{k=0}^{N-1}\sum_{j=0}^{N-1}\Big(\frac{(j+1)^{2H}-j^{2H}}{N^{2H}}+\frac{1}{N}\Big)\Big(\frac{(k+1)^{2H}-k^{2H}}{N^{2H}}+\frac{1}{N}\Big)E(M_{\frac{j+1}{N}}-M_{\frac{j}{N}})E(M_{\frac{k+1}{N}}-M_{\frac{k}{N}})
\end{eqnarray*}
\begin{eqnarray*}
&=&\sigma^6\sum_{j=0}^{N-1}\Big(\frac{(j+1)^{2H}-j^{2H}}{N^{2H}}+\frac{1}{N}\Big)^2 \times [\frac{1}{N^{2H}}+\frac{1}{N}]\\
&&\qquad +\sigma^6 \mathop{\disp\sum_{k=0}^{N-1}\sum_{j=0}^{N-1}}_{k\ne j}\Big(\frac{(j+1)^{2H}-j^{2H}}{N^{2H}}+\frac{1}{N}\Big)\Big(\frac{(k+1)^{2H}-k^{2H}}{N^{2H}}+\frac{1}{N}\Big)\\
&& \qquad \qquad \qquad \qquad \times\frac{1}{2}\Big[|\frac{k-j+1}{N}|^{2H}+|\frac{k-j-1}{N}|^{2H}-2|\frac{k-j}{N}|^{2H} \Big]\\
&\le & 4\sigma^2 U_1[\frac{1}{N^{2H}}+\frac{1}{N}]+\frac{\sigma^6}{2}\mathop{\disp\sum_{k=0}^{N-1}\sum_{j=0}^{N-1}}_{k\ne j}\Big(\frac{2H(j+1)^{2H-1}}{N^{2H}}+\frac{1}{N}\Big)\Big(\frac{2H(k+1)^{2H-1}}{N^{2H}}+\frac{1}{N}\Big)\\
&& \qquad \qquad \qquad \qquad \qquad \qquad \qquad \qquad \times \Big[2\frac{1}{N^2}2H(2H-1)|\frac{k-j}{N}|^{2H-2}\Big]\\
&\le & 4\sigma^2 U_1[\frac{1}{N^{2H}}+\frac{1}{N}]+\frac{\sigma^6 2H(2H-1)}{N^{2H}}U_4
\end{eqnarray*}
Where 
\begin{eqnarray*}
U_4&=& \mathop{\disp\sum_{k=0}^{N-1}\sum_{j=0}^{N-1}}_{k\ne j}\Big(\frac{2H(j+1)^{2H-1}}{N^{2H}}+\frac{1}{N}\Big)\Big(\frac{2H(k+1)^{2H-1}}{N^{2H}}+\frac{1}{N}\Big)|k-j|^{2H-2}\\
&=& 2\sum_{k=1}^{N-1}k^{2H-2}\sum_{j=1}^{N-k}(\frac{2H j^{2H-1}}{N^{2H}}+\frac{1}{N})(\frac{2H (k+j)^{2H-1}}{N^{2H}}+\frac{1}{N})\\
&\le & 2\sum_{k=1}^{N-1}k^{2H-2}\sum_{j=1}^{N-k}(\frac{2H(k+j)^{2H-1}}{N^{2H}}+\frac{1}{N})^2\\
&=& 2\sum_{k=1}^{N-1}k^{2H-2}\sum_{j=k+1}^{N}(\frac{2H j^{2H-1}}{N^{2H}}+\frac{1}{N})^2\\
&=& 2\sum_{k=1}^{N-1}k^{2H-2}\sum_{j=1}^{N}(\frac{2H j^{2H-1}}{N^{2H}}+\frac{1}{N})^2\\
&\le & 2\sum_{k=1}^{N-1}k^{2H-2} (4H^2 N^{-1}+4H N^{-1}+N^{-1})\\
&\le & 2N^{2H-1}(2H+1)^2 N^{-1}\\
&=& 2N^{2H-2}(2H+1)^2 
\end{eqnarray*}
So, $E(U_2^2)\rightarrow 0 \ \mbox{as} \ N\rightarrow \infty$.
The third term is
\begin{eqnarray*}
U_3 &=& \sum_{j=0}^{N-1} \sigma^2 (C_j+D_j)^2 
= \sum_{j=0}^{N-1} \sigma^2 (M_{\frac{j+1}{N}}-M_{\frac{j}{N}})^2
\end{eqnarray*} 
Now, $E(U_3)=\sigma^2 \disp\sum_{j=0}^{N-1} E(M_{\frac{j+1}{N}}-M_{\frac{j}{N}})^2=N\sigma^2[\disp\frac{1}{N^{2H}}+\frac{1}{N}]$.
Let us denote $$T_3=\disp\frac{1}{N}(\frac{1}{N^{2H}}+\frac{1}{N})^{-1}U_3$$ so $E(T_3)=\sigma^2$ and define $S_3=T_3-\sigma^2$ so that $E(S_3)=0$. 
 
Let us write $(M_{\frac{j+1}{N}}-M_{\frac{j}{N}})^2$ as multiple integral of order 2 with respect to isonormal Gaussian process $M$ corresponding to MFBM $M_t$. We have
\begin{eqnarray*}
M_{\frac{j+1}{N}}-M_{\frac{j}{N}}= I_1(f_j)
\end{eqnarray*}
where, $f_j(s)=1_{(\frac{j}{N},\frac{j+1}{N}]}(s)$. Using the product formula for multiple Wiener Ito intergral by (\ref{tp}) (for details see Nourdin (2012))
\begin{eqnarray*}
(M_{\frac{j+1}{N}}-M_{\frac{j}{N}})^2&=& (I_1(f_j))^2
=[I_2(f_j\otimes_0f_j)+I_0(f_j\otimes_1 f_j)]
\end{eqnarray*}
\begin{eqnarray*}
(M_{\frac{j+1}{N}}-M_{\frac{j}{N}})(M_{\frac{k+1}{N}}-M_{\frac{k}{N}})&=& (I_1(f_j))(I_1(f_k))
=[I_2(f_j\otimes_0 f_k)+I_0(f_j\otimes_1 f_k)]
\end{eqnarray*}
and $$I_0(f_j\otimes_1 f_j)=\langle f_j,f_j\rangle_\mathcal{H}=\disp(\frac{1}{N}+\frac{1}{N^{2H}}) .$$ 
As $ \disp E(M_{\frac{j+1}{N}}-M_{\frac{j}{N}})^2=E[I_2(f_j\otimes_0f_j)]+(\frac{1}{N}+\frac{1}{N^{2H}})$ so we note that $$E [\disp\sum_{j=0}^{N-1}I_2 (f_j\otimes_0 f_j)]= 0.$$
So, substituting in the expression of $U_3$ we get
\begin{eqnarray*}
U_3&=& \sigma^2 \sum_{j=0}^{N-1}[I_2(f_j\otimes_0 f_j)+I_0(f_j\otimes_1 f_j)]=
\sigma^2 \sum_{j=0}^{N-1} [I_2 (f_j\otimes_0 f_j) +(\frac{1}{N}+\frac{1}{N^{2H}})]
\end{eqnarray*}
Let us denote $\disp V_3=U_3-\sigma^2 N (\frac{1}{N}+\frac{1}{N^{2H}})=\sigma^2 \sum_{j=0}^{N-1} I_2 (f_j\otimes_0 f_j) $. This leads $E(V_3)=0$ and we observe that $$S_3=\disp\frac{1}{N}(\frac{1}{N^{2H}}+\frac{1}{N})^{-1}V_3=\disp\frac{1}{N}(\frac{1}{N^{2H}}+\frac{1}{N})^{-1}\sigma^2 \sum_{j=0}^{N-1} I_2 (f_j\otimes_0 f_j).$$ We note that the expectation $ES_3=0$. Now $E(V_3^2)$ is equal to
\begin{eqnarray*}
&&\sigma^4 \sum_{k=0}^{N-1} \sum_{j=0}^{N-1}E \Big( [I_2 (f_j\otimes_0 f_j) ] [I_2 (f_k\otimes_0 f_k) ]\Big)\\
&=& \disp\sigma^4 [2 \sum_{k=0}^{N-1} \sum_{j=0}^{N-1}\langle(f_j\otimes_0 f_j),(f_k\otimes_0 f_k)\rangle_{\mathcal{H}^{\otimes 2}}]\\
&=& \sigma^4 [2 \sum_{k=0}^{N-1} \sum_{j=0}^{N-1} (\langle f_j,f_k\rangle_{\mathcal{H}})^2]\\
&=& \sigma^4 [2 \sum_{k=0}^{N-1} \sum_{j=0}^{N-1}\Big(E(M_{\frac{j+1}{N}}-M_{\frac{j}{N}})(M_{\frac{k+1}{N}}-M_{\frac{k}{N}})\Big)^2]\\
&=& \sigma^4[2N(\frac{1}{N^{2H}}+\frac{1}{N})^2+\mathop{\sum_{k=0}^{N-1} \sum_{j=0}^{N-1}}_{k\ne j}\frac{1}{2} (|\frac{k-j+1}{N}|^{2H}+|\frac{k-j-1}{N}|^{2H}-2|\frac{k-j}{N}|^{2H})^2]\\
\end{eqnarray*}
And further 
\begin{eqnarray*}
&&E(\sqrt{N}S_3)^2\\
&=&\frac{1}{N}(\frac{1}{N^{2H}}+\frac{1}{N})^{-2}E(V_3^2)\\
&\le & \sigma^4[2+(\frac{1}{N^{2H}}+\frac{1}{N})^{-2}\mathop{\sum_{k=0}^{N-1} \sum_{j=0}^{N-1}}_{k\ne j}\frac{1}{2N} (\frac{2H(2H-1)}{N^2}|\frac{k-j+1}{N}|^{2H-2})^2]\\
&\le & \sigma^4[2+(\frac{1}{N^{2H}}+\frac{1}{N})^{-2}\mathop{\sum_{k=1}^{N} \sum_{j=0}^{N-1}}_{k\ne j}\frac{1}{2N} (\frac{2H(2H-1)}{N^2}|\frac{k-j}{N}|^{2H-2})^2]\\
&\le & \sigma^4[2+(\frac{1}{N^{2H}}+\frac{1}{N})^{-2}\mathop{\sum_{k=0}^{N} \sum_{j=0}^{N}}_{k\ne j}\frac{1}{2N} (\frac{2H(2H-1)}{N^{2H}}|k-j|^{2H-2})^2]\\
&\le & \sigma^4[2+(\frac{1}{N^{2H}}+\frac{1}{N})^{-2}\frac{(2H(2H-1))^2}{2N^{1+4H}} \mathop{\sum_{k=0}^{N} \sum_{j=0}^{N}}_{k\ne j}|k-j|^{4H-4}]\\
&\le & \sigma^4[2+(\frac{1}{N^{2H}}+\frac{1}{N})^{-2}\frac{(2H(2H-1))^2}{2N^{1+4H}} \mathop{\sum_{k=-N}^{N} \sum_{j=k}^{N+k}}_{k\ne 0}|k|^{4H-4}]\\
&\le & \sigma^4[2+(\frac{1}{N^{2H}}+\frac{1}{N})^{-2}\frac{(2H(2H-1))^2}{2N^{4H}} \mathop{\sum_{k=-N}^{N} }_{k\ne 0}|k|^{4H-4}]\\
&\le & \sigma^4[2+(\frac{1}{N^{2H}}+\frac{1}{N})^{-2}\frac{(2H(2H-1))^2}{2N^{4H}} N^{4H-3}]
\end{eqnarray*}
So, we see that $\disp\lim_{N\rightarrow\infty}\frac{1}{N}(\frac{1}{N^{2H}}+\frac{1}{N})^{-2}E(V_3^2)=\disp\lim_{N\rightarrow\infty}E(\sqrt{N}S_3)^2 =c$ some constant. 
So now we can see why we should choose normalising constant as $\disp\frac{1}{N}(\frac{1}{N^{2H}}+\frac{1}{N})^{-1}$. 
This is interesting because for pure Brownian motion we do not need any normalising factor depending on $N$. Also for Pure fractional Brownian motion we need normalising factor depending on $N$ but with the restriction $H<\frac{3}{4}$ to ensure finiteness for similar variance term. But for MFBM, $ H\in(0,1)$, normalising constant turns out to be dependent on both $N$ and $H$.
Let $T_1,T_2,T_3$ are normalised $U_1,U_2,U_3$ respectively, i.e.
$\disp \frac{1}{N}(\frac{1}{N^{2H}}+\frac{1}{N})^{-1}U_k=T_k,\ k=1,2,3$
\subsection*{Appendix: Almost sure convergence $H\in(0,1)$}
$E(\disp\frac{1}{N}(\frac{1}{N^{2H}}+\frac{1}{N})^{-1}\disp \sum_{j=0}^{N-1}\Big(\log\frac{S_{\frac{j+1}{N}}}{S_{\frac{j}{N}}}\Big)^2) \rightarrow\sigma^2 $ as $N\ \rightarrow \ \infty$.
And also,

$Var(\disp\frac{1}{N}(\frac{1}{N^{2H}}+\frac{1}{N})^{-1}\disp \sum_{j=0}^{N-1}\Big(\log\frac{S_{\frac{j+1}{N}}}{S_{\frac{j}{N}}}\Big))^2=o_{\mathbb{P}}(1).$ 
So, $\hat\sigma^2$ converges in probability to $\sigma^2$. 

To get almost sure convergence let us use Chebyshev inequality, Cauchy Schwartz inequality and hypercontractivity property (\ref{hc}). We recall That $\disp U_1\leq \frac{c_{1H}}{N}, \ c_{1H}$ is constant depends on $H$ and $\sigma$, 
$$U_2=-\sigma^3\disp \sum_{j=0}^{N-1}\Big( \frac{(j+1)^{2H}-j^{2H}}{N^{2H}}+\frac{1}{N}\Big)I_1(f_j)$$ and lastly
$V_3=\disp \sigma^2 \sum_{j=0}^{N-1} I_2 (f_j\otimes_0 f_j) .$
And from previous calculation we have for constants $c_{2H},c_{3H} $ depends on $H$ and $\sigma$, $$\disp E(U_2^2)\leq \frac{c_{2H}}{N}(\frac{1}{N^{2H}}+\frac{1}{N})+\frac{c_{3H}}{N^2}=c_{2H}(\frac{1}{N^{2H+1}}+\frac{1}{N^2})+\frac{c_{3H}}{N^2}$$ and for constants $c_{4H},c_{5H} $ depends on $H$ and $\sigma$ $$\disp  E(V_3^2)\leq c_{4H}N(\frac{1}{N^{2H}}+\frac{1}{N})^2+ \frac{c_{5H}}{N^2}=c_{4H}(\frac{1}{N^{2H-\frac{1}{2}}}+\frac{1}{N^{\frac{1}{2}}})^2+ \frac{c_{5H}}{N^2}$$
\begin{eqnarray*}
\mbox{Now} &&\mathbb{P}(|\hat\sigma^2-\sigma^2|>\frac{1}{N^{\delta}})\\&\leq & N^{r \delta}E|\hat\sigma^2-\sigma^2|^r\\
&= & N^{r \delta}E|T_1+T_2+T_3-\sigma^2 |^r\\
&= & N^{r \delta}E|T_1+T_2+S_3 |^r\\
&= & N^{r \delta}\sum_{k_1+k_2+k_3=r}\frac{r!}{k_1!k_2!k_3!}E[T_1^{k_1}T_2^{k_2}S_3^{k_3}]\\
&= & N^{r \delta} \frac{1}{N^r}(\frac{1}{N^{2H}}+\frac{1}{N})^{-r}\sum_{k_1+k_2+k_3=r}\frac{r!}{k_1!k_2!k_3!}E[U_1^{k_1}U_2^{k_2}V_3^{k_3}] \\
&\leq & N^{r (\delta-1)} (\frac{1}{N^{2H}}+\frac{1}{N})^{-r}\sum_{k_1+k_2+k_3=r}\frac{r!}{k_1!k_2!k_3!}U_1^{k_1}[E[U_2^{2k_2}]E[V_3^{2k_3}]]^{\frac{1}{2}} \\
&\leq & N^{r (\delta-1)} (\frac{1}{N^{2H}}+\frac{1}{N})^{-r}\sum_{k_1+k_2+k_3=r}\frac{r!}{k_1!k_2!k_3!}U_1^{k_1}[(2k_2-1)^{k_2}(2k_3-1)^{2k_3}E[U_2^{2}]^{k_2}E[V_3^{2}]^{k_3}]^{\frac{1}{2}} \\
&\leq & N^{r (\delta-1)} (\frac{1}{N^{2H}}+\frac{1}{N})^{-r}\sum_{k_1+k_2+k_3=r}c_{k_1,k_2,k_3}U_1^{k_1}E[U_2^{2}]^{k_2}E[V_3^{2}]^{k_3}]^{\frac{1}{2}},\ \ c_{k_1,k_2,k_3} \ \mbox{is constant} \\
&\leq & N^{r (\delta-1)} (\frac{1}{N^{2H}}+\frac{1}{N})^{-r}\sum_{k_1+k_2+k_3=r}c'_{k_1,k_2,k_3}\frac{1}{N^{k_1}}(\frac{1}{N^{2H+1}}+\frac{1}{N^2})^{\frac{k_2}{2}}(\frac{1}{N^{4H-1}}+\frac{1}{N^{2H}}+\frac{1}{N})^{\frac{k_3}{2}}
\end{eqnarray*}
Let us calculate the above probability for $\disp H\in(0,\frac{1}{4})$, so
\begin{eqnarray*}
&&\mathbb{P}(|\hat\sigma^2-\sigma^2|>\frac{1}{N^{\delta}})\\
&\leq & N^{r (\delta-1)} (\frac{1}{N})^{-r}\sum_{k_1+k_2+k_3=r}c'_{k_1,k_2,k_3}\frac{1}{N^{k_1}}(\frac{1}{N^{4H-1}})^{\frac{k_2}{2}}(\frac{1}{N^{2H+1}})^{\frac{k_3}{2}}\\ 
&= & N^{r (\delta-1)} N^r\sum_{k_1+k_2+k_3=r}c'_{k_1,k_2,k_3}N^{-k_1-2k_2H+\frac{k_2}{2}-k_3H-\frac{k_3}{2}}\\
&\leq & N^{r (\delta-1)} \sum_{k_1+k_2+k_3=r}c'_{k_1,k_2,k_3}N^{\frac{3k_2}{2}+\frac{k_3}{2}-2H(r-k_1)}\\ 
&\leq & N^{r (\delta-1)} \sum_{k_1+k_2+k_3=r}c'_{k_1,k_2,k_3}N^{\frac{1}{2}(r-k_1)-2H(r-k_1)}\\ 
&= & N^{r (\delta-1)} \sum_{k_1+k_2+k_3=r}c'_{k_1,k_2,k_3}N^{r(\frac{1}{2}-2H)+k_1(2H-\frac{1}{2})}\\ 
&\leq & N^{r (\delta-1)} \sum_{k_1+k_2+k_3=r}c'_{k_1,k_2,k_3}N^{r(\frac{1}{2}-2H)} 
\end{eqnarray*}
We need $\disp r(\delta-1)+r(\frac{1}{2}-2H)<-1\Rightarrow \ r>\frac{1}{2H+\frac{1}{2}-\delta}$. So for $\disp H\in(0,\frac{1}{4})$ fix $\disp \delta<2H+\frac{1}{2}$, then for $r$ larger than $\disp \frac{1}{2H+\frac{1}{2}-\delta}$, $\disp \sum_{N=1}^{\infty}\mathbb{P}(|\hat\sigma^2-\sigma^2|>\frac{1}{N^{\delta}})<\infty.$
 
Let us calculate the same probability for $\disp H\in[\frac{1}{4},\frac{1}{2})$, so
\begin{eqnarray*}
&&\mathbb{P}(|\hat\sigma^2-\sigma^2|>\frac{1}{N^{\delta}})\\
&\leq & N^{r (\delta-1)} (\frac{1}{N})^{-r}\sum_{k_1+k_2+k_3=r}c'_{k_1,k_2,k_3}\frac{1}{N^{k_1}}(\frac{1}{N^{2H}})^{\frac{k_2}{2}}(\frac{1}{N^{2H+1}})^{\frac{k_3}{2}}\\ 
&=& N^{r (\delta-1)} N^r\sum_{k_1+k_2+k_3=r}c'_{k_1,k_2,k_3}N^{-k_1-Hk_2-Hk_3-\frac{k_3}{2}}\\
&\leq & N^{r (\delta-1)} \sum_{k_1+k_2+k_3=r}c'_{k_1,k_2,k_3}N^{r-H(r-k_1)-k_1}\\
&\leq & N^{r (\delta-1)} \sum_{k_1+k_2+k_3=r}c'_{k_1,k_2,k_3}N^{(1-H)(r-k_1)}\\
&\leq & N^{r (\delta-1)} \sum_{k_1+k_2+k_3=r}c'_{k_1,k_2,k_3}N^{r(1-H)}
\end{eqnarray*}
Here we need $\disp r(\delta-1)+r(1-H)<-1\Rightarrow\ r>\frac{1}{H-\delta}$. So for $\disp H\in[\frac{1}{4},\frac{1}{2})$ fix $\disp \delta <H$, then for $r$ larger than $\disp \frac{1}{H-\delta}$, $\disp \sum_{N=1}^{\infty}\mathbb{P}(|\hat\sigma^2-\sigma^2|>\frac{1}{N^{\delta}})<\infty.$

Let us again calculate the same probability for $\disp H\in[\frac{1}{2},1)$.
\begin{eqnarray*}
&&\mathbb{P}(|\hat\sigma^2-\sigma^2|>\frac{1}{N^{\delta}})\\
&\leq & N^{r (\delta-1)} (\frac{1}{N^{2H}}+\frac{1}{N})^{-r}\sum_{k_1+k_2+k_3=r}c'_{k_1,k_2,k_3}\frac{1}{N^{k_1}}(\frac{1}{N})^{\frac{k_2}{2}}(\frac{1}{N^2})^{{\frac{k_3}{2}}}\\
&= & N^{r (\delta-1)} (\frac{1}{N^{2H}}+\frac{1}{N})^{-r}\sum_{k_1+k_2+k_3=r}c'_{k_1,k_2,k_3}N^{-k_1-\frac{k_2}{2}-k_3}\\
&= & N^{r (\delta-1)} (\frac{1}{N^{2H}}+\frac{1}{N})^{-r}\sum_{k_1+k_2+k_3=r}c'_{k_1,k_2,k_3}N^{-r+\frac{k_2}{2}}\\
&= & N^{r (\delta-1)} (\frac{1}{N^{2H}}+\frac{1}{N})^{-r}\sum_{k_1+k_2+k_3=r}c'_{k_1,k_2,k_3}N^{-\frac{r}{2}}\\
&= & (N^{\frac{3}{2}-2H-\delta}+N^{\frac{1}{2}-\delta})^{-r}\sum_{k_1+k_2+k_3=r}c'_{k_1,k_2,k_3}
\end{eqnarray*}
For $\disp H\in[\frac{1}{2},1)$ fix $\disp \delta<\frac{1}{2}$ and choose $\disp r>\frac{1}{\frac{1}{2}-\delta}$,  $\disp \sum_{N=1}^{\infty}\mathbb{P}(|\hat\sigma^2-\sigma^2|>\frac{1}{N^{\delta}})<\infty.$ We apply Borel Cantelli lemma above all cases to get almost sure convergence.

\subsection*{Appendix: Asymptotic normality}
We recall that, $E(\sqrt{N}[\hat{\sigma}^2-\sigma^2]) \rightarrow c$ as $N\rightarrow \infty$. Let us  introduce the notation. 
$\disp F_N:=\sqrt{N}\frac{(\hat\sigma^2-\sigma^2)}{\sqrt{c}}$. Then $$F_N=\disp \frac{\sqrt{N}}{\sqrt{c}}[\frac{1}{N}(\frac{1}{N^{2H}}+\frac{1}{N})^{-1}(U_1+U_2+U_3)-\sigma^2]=\frac{\sqrt{N}}{\sqrt{c}}(T_1+T_2+T_3-\sigma^2).$$  From our previous calculations we see that $\disp\frac{\sqrt{N}}{\sqrt{c}}T_1\rightarrow 0$ and $\disp\frac{\sqrt{N}}{\sqrt{c}}T_2\rightarrow 0$ in $L_2$. So here we will show $$G_N:=\disp \frac{\sqrt{N}}{\sqrt{c}}(T_3-\sigma^2)=\disp \frac{\sqrt{N}}{\sqrt{c}}S_3$$ is asymptotically normally distributed and then by Slutsky's theorem $F_N$ will be asymptotically normal. To show $G_N$ is asymptotically normal we will use Theorem \ref{md} and \ref{fm}.

We recall $E(T_3)=\sigma^2,\ E(G_N)=0,\ E(F_N)=0$.
Also we note that $E(G_N^2)\rightarrow\ 1 $ as $N\rightarrow\ \infty$ and so $E(F_N^2)\rightarrow\ 1 $.

Now let us recall $T_3$ as multiple Wiener Ito integral
\begin{eqnarray*}
T_3= \disp \frac{1}{N}(\frac{1}{N^{2H}}+\frac{1}{N})^{-1} \sigma^2\sum_{j=0}^{N-1}[I_2(f_j\otimes_0 f_j)+(\frac{1}{N}+\frac{1}{N^{2H}})]
\end{eqnarray*} 
So $S_3$ as multiple Wiener Ito integral as
\begin{eqnarray*}
S_3= \disp \frac{1}{N}(\frac{1}{N^{2H}}+\frac{1}{N})^{-1} \sigma^2\sum_{j=0}^{N-1}[I_2(f_j\otimes_0 f_j)]
\end{eqnarray*} 

Using theorems stated in section (\ref{note}) we want to show that $\|DG_N\|^2_{\mathcal{H}}\rightarrow 2$ in $L^2$. For that matter we first show 
$\disp\lim_{N\rightarrow \infty}E[\|DG_N\|^2_{\mathcal{H}}]= 2$ and then $\disp\lim_{N\rightarrow \infty}E[\|DG_N\|^2_{\mathcal{H}}-2]^2=0$. Now, 
\begin{eqnarray}
&&[\|DG_N\|^2_{\mathcal{H}}-2]^2\\
&=&[\|DG_N\|^2_{\mathcal{H}}-E[\|DG_N\|^2_{\mathcal{H}}]+E[\|DG_N\|^2_{\mathcal{H}}]-2]^2\\
&=&[\|DG_N\|^2_{\mathcal{H}}-E[\|DG_N\|^2_{\mathcal{H}}]^2+[E[\|DG_N\|^2_{\mathcal{H}}]-2]^2\\
&&\qquad \qquad +2[\|DG_N\|^2_{\mathcal{H}}-E[\|DG_N\|^2_{\mathcal{H}}][E[\|DG_N\|^2_{\mathcal{H}}]-2]\\
&=&A+B+2C \label{varbias}
\end{eqnarray}
where $A, B, C$ are respective terms. Now $EC=0, B\rightarrow 0$ as $ N\rightarrow\infty$. For most purpose we will be interested in $A$.

Using (\ref{dd}) we get $$ \disp D_{t}S_3=2 \frac{1}{N}(\frac{1}{N^{2H}}+\frac{1}{N})^{-1}\sigma^2 \sum_{j=0}^{N-1}f_j(t)I_1(f_j).$$ So $$ \|DS_3\|_{\mathcal{H}}^2=\disp \frac{1}{N^2}(\frac{1}{N^{2H}}+\frac{1}{N})^{-2}4\sigma^4\sum_{k=0}^{N-1}\sum_{j=0}^{N-1} I_1(f_j)I_1(f_k)\langle f_j,f_k\rangle_{\mathcal{H}}.$$
Now $$E[\|DS_3\|_{\mathcal{H}}^2]=\disp \frac{1}{N^2}(\frac{1}{N^{2H}}+\frac{1}{N})^{-2}4\sigma^4\sum_{k=0}^{N-1}\sum_{j=0}^{N-1}(\langle f_j,f_k\rangle_{\mathcal{H}})^2$$ and $\disp\|DG_N\|=\frac{\sqrt{N}}{\sqrt{c}}\|DS_3\|_{\mathcal{H}}, \ \ \ E\|DG_N\|=\frac{\sqrt{N}}{\sqrt{c}}E(\|DS_3\|_{\mathcal{H}}) $

We have previously shown that $\disp E(\sqrt{N}S_3)^2 =c=2\sigma^4$, and similar calculation leads to $\disp \lim_{N\rightarrow \infty} E[\|DS_3\|_{\mathcal{H}}^2]=2 \disp \lim_{N\rightarrow \infty}E(\sqrt{N}S_3)^2=c$. 
So $E[\|DG_N\|_{\mathcal{H}}^2]\rightarrow 2$ as $N\rightarrow\infty$. 
Let us calculate the following
\begin{eqnarray*}
&&\|DS_3\|_{\mathcal{H}}^2-E[\|DS_3\|_{\mathcal{H}}^2]\\
&=&\disp \frac{1}{N^2}(\frac{1}{N^{2H}}+\frac{1}{N})^{-2}4\sigma^4\sum_{k=0}^{N-1}\sum_{j=0}^{N-1}[\left( I_2(f_j\otimes_0 f_k)+I_0(f_j\otimes_1 f_k)\right)\langle f_j,f_k\rangle_{\mathcal{H}}-\langle f_j,f_k\rangle_{\mathcal{H}}^2]\\
&=& \disp \frac{1}{N^2}(\frac{1}{N^{2H}}+\frac{1}{N})^{-2}4\sigma^4\sum_{k=0}^{N-1}\sum_{j=0}^{N-1}[ I_2(f_j\otimes_0 f_k)\langle f_j,f_k\rangle_{\mathcal{H}}]
\end{eqnarray*}

And then 
\begin{eqnarray}
&&E[\|DS_3\|_{\mathcal{H}}^2-E[\|DS_3\|_{\mathcal{H}}^2]]^2\\
&=&16\sigma^{8} \frac{1}{N^4}(\frac{1}{N^{2H}}+\frac{1}{N})^{-4} E[\sum_{k=0}^{N-1}\sum_{j=0}^{N-1}[ I_2(f_j\otimes f_k)\langle f_j,f_k\rangle_{\mathcal{H}}]]^2\\
&=&16\sigma^{8} \frac{1}{N^4}(\frac{1}{N^{2H}}+\frac{1}{N})^{-4} \sum_{k,k'=0}^{N-1}\sum_{j,j'=0}^{N-1}\langle f_j,f_k\rangle_{\mathcal{H}}\langle f_{j'},f_{k'}\rangle_{\mathcal{H}}\langle f_j,f_{j'}\rangle_{\mathcal{H}}\langle f_k,f_{k'}\rangle_{\mathcal{H}}
\label{4m}\end{eqnarray}

So, our target quantity is \begin{equation}E[\|DG_N\|_{\mathcal{H}}^2-E[\|DG_N\|_{\mathcal{H}}^2]]^2=\disp \frac{N^2}{c^2}E[\|DS_3\|_{\mathcal{H}}^2-E[\|DS_3\|_{\mathcal{H}}^2]]^2 .
\end{equation}

Let us denote the adjusted normalising constant, adjusted for the terms $$\disp \sum_{k,k'=0}^{N-1}\sum_{j,j'=0}^{N-1}\langle f_j,f_k\rangle_{\mathcal{H}}\langle f_{j'},f_{k'}\rangle_{\mathcal{H}}\langle f_j,f_{j'}\rangle_{\mathcal{H}}\langle f_k,f_{k'}\rangle_{\mathcal{H}}
$$ in target quantity as $K$ with  $\disp K^{-1}=\frac{1}{N^2}(\frac{1}{N^{2H}}+\frac{1}{N})^{-4}$, ignoring the constant $\disp\frac{16\sigma^{8}}{c^2}=4$.

For $j=k$,  $\langle f_j,f_k\rangle_{\mathcal{H}}=I_0(f_j\otimes f_j)= \disp (\frac{1}{N}+\frac{1}{N^{2H}})$;  for $|k-j|= 1$,  $\langle f_j,f_k\rangle_{\mathcal{H}}=I_0(f_j\otimes_1 f_k)\le |2^{2H}-2| \disp \frac{1}{N^{2H}}$ and for $|k-j|> 1$,  $\langle f_j,f_k\rangle_{\mathcal{H}}=I_0(f_j\otimes_1 f_k) \le \disp\frac{|2H(2H-1)||j-k|^{2H-2}}{N^{2H}}$.

To calculate contribution from terms in (\ref{4m}) let us introduce partition of $4$ as $$\{\{4\},\{3,1\},\{2,2\},\{2,1,1\},\{1,1,1,1\}\}.$$ This is required as there are four indices $j,j',k,k'$ in the summation (\ref{4m}).

{\bf Case : all index same}

For i) $j=k=j'=k'$, there are $N$ terms in the sum and value of$$\langle f_j,f_k\rangle_{\mathcal{H}}\langle f_{j'},f_{k'}\rangle_{\mathcal{H}}\langle f_j,f_{j'}\rangle_{\mathcal{H}}\langle f_k,f_{k'}\rangle_{\mathcal{H}} =\disp(\frac{1}{N}+\frac{1}{N^{2H}})^4.$$ So total contribution due to i) is $\disp N(\frac{1}{N}+\frac{1}{N^{2H}})^4$. 

{\bf Case : three index same, one different: one-apart}

For ii) $j=k=j'$ and $|k'-k|=1$, there are $2(N-1)$ terms in the sum and value of $$\langle f_j,f_k\rangle_{\mathcal{H}}\langle f_{j'},f_{k'}\rangle_{\mathcal{H}}\langle f_j,f_{j'}\rangle_{\mathcal{H}}\langle f_k,f_{k'}\rangle_{\mathcal{H}}=\disp(\frac{1}{N}+\frac{1}{N^{2H}})^2|2^{2H}-2|^2\frac{1}{N^{4H}}$$.

Also for iii) $j=k=k'$, $|j'-j|=1$, iv) $j'=k'=j$,  $|k-k'|=1$, v) $j'=k'=k$, $|j-j'|=1$ the contribution is same as that of the case $j=k=j'$ and $|k'-k|=1$.

{\bf Case : two pairs of same index: one-apart}

For vi) $j=k$, $j'=k'$ and $|j-j'|=1$, vii) $j=j'$, $k=k'$ and $|j-j'|=1$ have same contribution of the case $j=k=j'$ and $|k'-k|=1$.

So total contribution due to ii) to vii) is $12(N-1)\disp(\frac{1}{N}+\frac{1}{N^{2H}})^2|2^{2H}-2|^2\frac{1}{N^{4H}} $.
 
For viii) $j=k'$, $j'=k$ and $|j-j'|=1$ there are $2(N-1)$ terms in the sum and value of $\langle f_j,f_k\rangle_{\mathcal{H}}\langle f_{j'},f_{k'}\rangle_{\mathcal{H}}\langle f_j,f_{j'}\rangle_{\mathcal{H}}\langle f_k,f_{k'}\rangle_{\mathcal{H}}=\disp |2^{2H}-2|^4\frac{1}{N^{8H}}$. So total contribution due to viii) is $2(N-1)\disp |2^{2H}-2|^4\frac{1}{N^{8H}}$.

{\bf Case : three index same, one different: more than one apart}
 
For ix) $j=k=j'$ and $|k'-k|>1$ the sum becomes 
\begin{eqnarray*}
&& \sum_{j'=0}^{N-1}\sum_{j=k=j', |k'-k|>1}\langle f_j,f_k\rangle_{\mathcal{H}}\langle f_{j'},f_{k'}\rangle_{\mathcal{H}}\langle f_j,f_{j'}\rangle_{\mathcal{H}}\langle f_k,f_{k'}\rangle_{\mathcal{H}}\\&=&\sum_{j'=0}^{N-1}\sum_{j=k=j', |k'-k|>1}\disp(\frac{1}{N}+\frac{1}{N^{2H}})^2|2H(2H-1)|^2|j'-k'|^{2H-2}|k-k'|^{2H-2}\frac{1}{N^{4H}}\\
&=&\disp(\frac{1}{N}+\frac{1}{N^{2H}})^2\frac{|2H(2H-1)|^2}{N^{4H}} \sum_{j'=0}^{N-1}\sum_{j=k=j', |k'-k|>1}|j'-k'|^{2H-2}|k-k'|^{2H-2}\\
&=&\disp(\frac{1}{N}+\frac{1}{N^{2H}})^2\frac{|2H(2H-1)|^2}{N^{4H}} \mathop{\sum_{j'=0}^{N-1}\sum_{k'=0}^{N-1}}_{|j'-k'|>1}|j'-k'|^{4H-4}\\
&=&\disp(\frac{1}{N}+\frac{1}{N^{2H}})^2\frac{|2H(2H-1)|^2}{N^{4H}} N\mathop{\sum_{j'=0}^{N-1}}|j'|^{4H-4}\\
&=&\disp(\frac{1}{N}+\frac{1}{N^{2H}})^2\frac{|2H(2H-1)|^2}{N^{4H}} N^{4H-2}\\
&=&\disp(\frac{1}{N}+\frac{1}{N^{2H}})^2\frac{|2H(2H-1)|^2}{N^2}
\end{eqnarray*}
We note that above contribution is due to $(N^2-3N+2)$ many terms.

For x) $j=k=k'$, $|j'-j|>1$, xi) $j'=k'=j$,  $|k-k'|>1$, xii) $j'=k'=k$, $|j-j'|>1$ the contribution is same as that of the case $j=k=j'$ and $|k'-k|>1$, and each contribution is due to $(N^2-3N+2)$ many terms.

{\bf Case : two pairs of same index: more than one apart}

For xiii) $j=k$, $j'=k'$ and $|j-j'|>1$ xiv)$j=j'$, $k=k'$ and $|j-j'|>1$ the contribution is same as that of the case $j=k=j'$ and $|k'-k|>1$. Here each contribution is due to $(N^2-3N+2)$ many terms.

So total contribution due to ix) to xiv) is $6\disp\frac{1}{N^2}(\frac{1}{N}+\frac{1}{N^{2H}})^2|2H(2H-1)|^2$.

For xv) $j=k'$, $j'=k$ and $|j-j'|>1$ the sum becomes 
\begin{eqnarray*}
&& \mathop{\sum_{j=0}^{N-1}}_{j=k'}\mathop{\sum_{ j'=0}^{N-1}}_{j'=k,|j-j'|>1}\langle f_j,f_k\rangle_{\mathcal{H}}\langle f_{j'},f_{k'}\rangle_{\mathcal{H}}\langle f_j,f_{j'}\rangle_{\mathcal{H}}\langle f_k,f_{k'}\rangle_{\mathcal{H}}\\
&=&\mathop{\sum_{j=0}^{N-1}}_{j=k'}\mathop{\sum_{ j'=0}^{N-1}}_{j'=k,|j-j'|>1}\disp|2H(2H-1)|^4|j-k|^{2H-2}|k-k'|^{2H-2}|j-j'|^{2H-2}|k-k'|^{2H-2}\frac{1}{N^{8H}}\\
&=&\sum_{j=0}^{N-1}\sum_{|j'-j|>1, j'=0}^{N-1}\disp|2H(2H-1)|^4|j-j'|^{8H-8}\frac{1}{N^{8H}}\\
&= &\frac{|2H(2H-1)|^4}{N^{8H}}\sum_{j=0}^{N-1}\sum_{|j'-j|>1, j'=0}^{N-1}|j-j'|^{8H-8}\\
&=& \frac{|2H(2H-1)|^4}{N^{8H}}N\sum_{j=0}^{N-1}|j|^{8H-8}\\
&=& \frac{|2H(2H-1)|^4}{N^{8H}}N^{8H-6}\\
&=& \frac{|2H(2H-1)|^4}{N^{6}}
\end{eqnarray*}

From now onwards let us consider the cases but we will calculate the contribution corresponding to those at the end.

{\bf Case : two same index, third index: one apart, fourth index:  one apart}

There are $8$ cases like following have same contribution.

$j=k$, $|j-j'|=1$ and $|j-k'|=2, |j'-k'|=1$ 

There are $4$ cases like below have same contribution.

$j=k'$, $|j-j'|=1$ and $|j-k|=2, |j'-k'|=1$ 

{\bf Case : two same index, third index: one apart, fourth index: more than one apart}

There are $8$ cases like below.

$j=k$, $|j-j'|=1$ and $|k-k'|>1, |j'-k'|=1$ 

$j=k$, $|k-k'|=1$ and $|j-j'|>1, |j'-k'|=1$ 

There are $4$ cases like below.

$j=k'$, $|j-j'|=1$ and $|k-k'|>1, |j'-k'|=1$ 

$j=k'$, $|k-k'|=1$ and $|j-j'|>1, |j'-k'|=1$ 

{\bf Case : two same index, third index: more than one apart, fourth index: more than one apart}

There are $8$ cases like following have same contribution.

For $j=k$, $|j-j'|=1$ and $|k-k'|>1, |j'-k'|>1$

There are $4$ cases like below have same contribution.

$j=k'$, $|j-j'|=1$ and $|j-k|>1, |j'-k'|>1$ 

{\bf Case : all four indices / indexes different}

There are several (but finite) cases as follows: all four one apart, three of them one apart and one of them more than one apart, two pairs are more than one apart and each pair one apart, two of them one apart and third and fourth both more than one apart and lastly all fore more than one apart. Let us note that contribution from the last case will be highest and all other cases mentioned above have contribution less than that. As there are finitely many cases it is enough to consider the following case.

For $j\neq k \neq j' \neq k' $ and all are more than one units apart the the sum becomes 
\begin{eqnarray*}
&&\sum_{j=0}^{N-1}\mathop{\sum_{j'=0}^{N-1}}_{|j'-j|>2}\mathop{\sum_{k=0}^{N-1}}_{|k-j|>2,|k-j'|>2}\mathop{\sum_{ k'=0}^{N-1}}_{|k'-j'|>2,|k'-k|>2,|k'-j|>2}\langle f_j,f_k\rangle_{\mathcal{H}}\langle f_{j'},f_{k'}\rangle_{\mathcal{H}}\langle f_j,f_{j'}\rangle_{\mathcal{H}}\langle f_k,f_{k'}\rangle_{\mathcal{H}}\\
&=&\mathop{\sum_{j,j',k,k'=0}^{N-1}}_{with \ restriction}\disp|2H(2H-1)|^4\frac{1}{N^{8H}}|j-k|^{2H-2}|j'-k'|^{2H-2}|j-j'|^{2H-2}|k-k'|^{2H-2}\\
&\leq & \disp\frac{|2H(2H-1)|^4}{N^{8H}}N^{4(2H-1)}\\
&=&\disp|2H(2H-1)|^4\frac{1}{N^4}
\end{eqnarray*}

Now we observe that for each of the above cases the contribution due to $$\disp \mathop{\sum_{k,k'=0}^{N-1}\sum_{j,j'=0}^{N-1}}_{with\ restriction}\langle f_j,f_k\rangle_{\mathcal{H}}\langle f_{j'},f_{k'}\rangle_{\mathcal{H}}\langle f_j,f_{j'}\rangle_{\mathcal{H}}\langle f_k,f_{k'}\rangle_{\mathcal{H}}$$ when divided by $K$ tends to zero. As an example let us show this calculation for last case.
\begin{eqnarray*}
&&\frac{1}{K}|2H(2H-1)|^4\frac{1}{N^4}\\
&=& (\frac{1}{N}+\frac{1}{N^{2H}})^{-4}\frac{1}{N^6}|2H(2H-1)|^4\\
&=& |2H(2H-1)|^4(N^{\frac{6}{4}}(\frac{1}{N}+\frac{1}{N^{2H}}))^{-4}\\
&=& |2H(2H-1)|^4(N^{\frac{1}{2}}+N^{\frac{3}{2}-2H})^{-4}\rightarrow \ 0 \ \mbox{as} \ N\rightarrow\infty
\end{eqnarray*}
As there are finitely many cases and each normalised term tends to zero. So we get the asymptotic normality for the estimator.

\subsection*{Berry Esseen bound}
First recall that $E(G_N)=0$ and $E(G_N^2)\rightarrow 1$ as $N\rightarrow\infty$.
Let us calculate similar quantity $\psi(N)$ as in the theorem \ref{ou} for $\disp G_N=\frac{1}{\sqrt{2N}}(\frac{1}{N}+\frac{1}{N^{2H}})^{-1}\sum_{j=0}^{N-1}I_2(f_j\otimes_0 f_j)$. Now we calculate 
\begin{eqnarray*}
&&(1-\langle DG_N, -DL^{-1}G_N\rangle_{\mathcal{H}})^2\\
&=&(1-\frac{1}{2}\|DG_N\|^2_{\mathcal{H}} )^2\\
&=& \frac{1}{4}E[2-\|DG_N\|^2_{\mathcal{H}} ]^2\\
&=& \frac{1}{4}[E(A)+E(B)] \ \mbox{using (\ref{varbias})}\\
\end{eqnarray*}

Then $\disp E(1-\langle DG_N, -DL^{-1}G_N\rangle_{\mathcal{H}})^2\approx \frac{1}{N}$, from calculation in previous section, dominant term contribution coming from the case: all index same and after dividing by $K$. 

So, $ \disp\psi(N)=\sqrt{E(1-\langle DG_N, -DL^{-1}G_N\rangle_{\mathcal{H}})^2}\approx\frac{1}{\sqrt{N}}\rightarrow 0$ satisfies condition (b) in Theorem \ref{be}.

Let $\disp R_N=\frac{1-\langle DG_N,-DL^{-1}G_N\rangle_{\mathcal{H}}}{\psi(N)}$. Using calculation from previous section $E(R_N)\rightarrow 0$ and $E(R_N^2)\rightarrow 1$ as $N\rightarrow\infty$.

Let us calculate covariance $\rho$ as in Theorem \ref{be}.
For that purpose recall $$
R_N=\frac{2-\|DG_N\|_{\mathcal{H}}^2}{2\psi(N)}
=\frac{\sqrt{N}}{2}[2-E(\|DG_N\|_{\mathcal{H}}^2)+E(\|DG_N\|_{\mathcal{H}}^2)-\|DG_N\|_{\mathcal{H}}^2]
$$
and $$\|DG_N\|_{\mathcal{H}}^2-E(\|DG_N\|_{\mathcal{H}}^2)=\frac{2}{N}(\frac{1}{N}+\frac{1}{N^{2H}})^{-2}\sum_{j=0}^{N-1}\sum_{k=0}^{N-1}I_2(f_j\otimes_0 f_k)\langle f_j,f_k\rangle_{\mathcal{H}} $$

\begin{eqnarray*}
\rho&=& \lim_{N\rightarrow\infty}E(G_N R_N)\\
&=& \lim_{N\rightarrow\infty} E(G_N \frac{\sqrt{N}}{2}[E(\|DG_N\|_{\mathcal{H}}^2)-\|DG_N\|_{\mathcal{H}}^2]\\
&=&-\frac{1}{\sqrt{2}N}(\frac{1}{N}+\frac{1}{N^{2H}})^{-3}E[\sum_{i=0}^{N-1}  I_2(f_i\otimes_0 f_i)   \sum_{j=0}^{N-1}\sum_{k=0}^{N-1}I_2(f_j\otimes_0 f_k)\langle f_j,f_k\rangle_{\mathcal{H}}]
\end{eqnarray*}
Using $E[I_2(f_i\otimes_0 f_j)I_2(f_k\otimes_0 f_l)=\langle f_i,f_k\rangle_{\mathcal{H}}\langle f_j,f_l\rangle_{\mathcal{H}}+\langle f_j,f_k\rangle_{\mathcal{H}}\langle f_i,f_l\rangle_{\mathcal{H}}$ we can calculate $\rho$ explicitly.

Now maximum contribution from the sum above when $i=j=k$ and that is 

$\disp\sum_{j=0}^{N-1}\langle f_j,f_j\rangle_{\mathcal{H}}^3=N(\frac{1}{N}+\frac{1}{N^{2H}})^3$. After multiplying by  $\disp -\frac{1}{\sqrt{2}N}(\frac{1}{N}+\frac{1}{N^{2H}})^{-3}$ it becomes $-\frac{1}{\sqrt{2}}$.

The minimum contribution from the sum above when $i\ne j\ne k$ and that is 

$\disp\sum_{i=0}^{N-1}\sum_{j=0}^{N-1}\sum_{k=0}^{N-1}\langle f_i,f_j\rangle_{\mathcal{H}}\langle f_j,f_k\rangle_{\mathcal{H}}\langle f_j,f_k\rangle_{\mathcal{H}}=\disp\sum_{i=0}^{N-1}\sum_{j=0}^{N-1}\sum_{k=0}^{N-1}c_H\frac{(|i-j||j-k||k-i|)^{2H-2}}{N^{6H}}=\frac{1}{N^3}$, $c_H$ being constant depends on $H$. After multiplying by  $\disp -\frac{1}{\sqrt{2}N}(\frac{1}{N}+\frac{1}{N^{2H}})^{-3}$ and taking limit it becomes $0$.
All other cases (finitely many) also the final contribution turns out to be $0$. Therefore $\rho=-\frac{1}{\sqrt{2}}$.

Using Theorem \ref{be} $\disp \sqrt{N}(P( G_N \leq x)-\Phi(x))\rightarrow -\frac{\Phi^{(3)}(x)}{3\sqrt{2}}$ as $N\rightarrow\infty$.

To get the result for $F_N$ we note that $F_N=I_2+I_1+I_0$ where $I_n$ is multiple integral of order $n$. Using $DI_1$ is $I_0$, $D$ being Malliavin derivative; correlation between $I_0,I_1$, $I_0,I_2$ and $I_0,I_2$ are zero;  and result we have proved for$G_N$ which is $I_2$; we can get Berry Esseen bound for $F_N$ and hence the result.

\section{Acknowledgement}
Author wants to acknowledge Department of Science and Technology, India, for financial support to conduct this research work.

\end{document}